\documentclass[12pt]{article}
\usepackage{amsmath,amssymb}    
\usepackage{hyperref}
\usepackage{graphicx}
\makeatletter
\newfont{\footsc}{cmcsc10 at 8truept}
\newfont{\footbf}{cmbx10 at 8truept}  
\newfont{\footrm}{cmr10 at 10truept} 
\makeatother
\def\stf#1#2{\left[#1\atop#2\right]} 
\def\sts#1#2{\left\{#1\atop#2\right\}}

\newtheorem{theorem}{Theorem}
\newtheorem{Prop}{Proposition}
\newtheorem{Cor}{Corollary}
\newtheorem{Lem}{Lemma}

\def \C{\mathbb{C}}

\newcommand{\qed}{\hfill \rule{0.7ex}{1.5ex}}

\newenvironment{proof}{\begin{trivlist} \item[\hskip \labelsep{\it
Proof.}]\setlength{\parindent}{0pt}}{\end{trivlist}}

\makeatletter
\newcommand{\subjclass}[2][2010]{%
  \let\@oldtitle\@title%
  \gdef\@title{\@oldtitle\footnotetext{#1 \emph{Mathematics subject classification.} #2}}%
}
\newcommand{\keywords}[1]{%
  \let\@@oldtitle\@title%
  \gdef\@title{\@@oldtitle\footnotetext{\emph{Key words and phrases.} #1.}}%
}
\makeatother
\title{Zeta functions interpolating the convolution of the Bernoulli polynomials}

\author{
Abdelmejid Bayad\\
\small D\'epartement de Math\'ematiques\\[-0.8ex]
\small Universit\'e d'Evry Val d'Essonne\\[-0.8ex]
\small 23 Bd. de France, 91037 Evry Cedex, France\\[-0.8ex]
\small \texttt{abayad@maths.univ-evry.fr}\\\\
Takao Komatsu
\\   
\small School of Mathematics and Statistics\\[-0.8ex]
\small Wuhan University, Wuhan, 430072, China\\[-0.8ex]
\small \texttt{komatsu@whu.edu.cn}
}

\date{
}
\begin{document}

\maketitle

\begin{abstract}  
We prove nonlinear relation on multiple Hurwitz-Riemann zeta functions. Using analytic continuation of these multiple Hurwitz-Riemann zeta function, we quote at negative integers  Euler's nonlinear relation for generalized Bernoulli polynomials and numbers.
\end{abstract}
\subjclass[2]{****xxx}
\section{Introduction and Preliminaries} 

Let 
$$
f(x)=\frac{1}{e^x-1}
$$
and 
$$
b(x)=x f(x)=\frac{x}{e^x-1}=\sum_{n=0}^\infty B_n\frac{x^n}{n!}\quad(|x|<1)\,,  
$$ 
where $B_n$ is the $n$-th Bernoulli number.  
Then we have the formulae: 

\begin{Lem} 
For $n\ge 0$ 
\begin{equation} 
f^{(n)}(x)=(-1)^n\sum_{k=1}^{n+1}(k-1)!\sts{n+1}{k}f(x)^k 
\label{11} 
\end{equation} 
and 
\begin{equation} 
f(x)^n=\frac{(-1)^{n-1}}{(n-1)!}\sum_{k=0}^{n-1}\stf{n}{k+1}f^{(k)}(x)\,, 
\label{12} 
\end{equation} 
where $\stf{n}{k}$ denotes the (unsigned) Stirling number of the first kind and $\sts{n}{k}$ denotes the Stirling number of the second kind.  
\end{Lem} 

\begin{proof}

The formulae (\ref{11}) and (\ref{12}) are equivalent via the orthogonal relation satisfied by Stirling numbers 
\begin{equation} 
\sum_{k=0}^{\max\{n, m\}}(-1)^k\stf{n}{k}\sts{k}{m}=(-1)^n\delta_{n,m}\,. 
\label{13} 
\end{equation} 
It is easy to see that both formulae are valid for $n=1$. Let $n\ge 2$. If the formula (\ref{11}) is true, then 
\begin{align*} 
&\frac{(-1)^{n-1}}{(n-1)!}\sum_{k=0}^{n-1}\stf{n}{k+1}f^{(k)}(x)\\
&=\frac{(-1)^{n-1}}{(n-1)!}\sum_{k=0}^{n-1}\stf{n}{k+1}(-1)^k\sum_{i=1}^{k+1}(i-1)!\sts{n+1}{i}f(x)^i\\
&=\frac{(-1)^{n-1}}{(n-1)!}\sum_{i=1}^n(i-1)!f(x)^i\sum_{k=i-1}^{n-1}(-1)^k\stf{n}{k+1}\sts{n+1}{i}\\
&=\frac{(-1)^{n-1}}{(n-1)!}\sum_{i=1}^n(i-1)!f(x)^i(-1)^{n-1}\delta_{n,i}\\
&=f(x)^n\,. 
\end{align*} 
On the other hand, the relation (\ref{11}) can be proven by induction on $n$ by using the recurrence formula 
$$
\sts{n}{m}=m\sts{n-1}{m}+\sts{n-1}{m-1}\quad(n,m\ge 1)\,.
$$
with 
$$
\sts{0}{0}=1,\quad \sts{n}{0}=\sts{0}{n}=0\quad(n\ge 1)\,. 
$$
\qed\end{proof}

Applying the Leibniz derivative formula to $b(x)=x f(x)$, we obtain 
$$
b^{(n)}(x)=x f^{(n)}(x)+n f^{(n-1)}(x)\,. 
$$ 
Then we have  
\begin{equation} 
b^{(n)}(x)=(-1)^n\sum_{k=1}^{n+1}(k-1)!\left(\sts{n+1}{k}x-n\sts{n}{k}\right)f(x)^k\,. 
\label{15}
\end{equation} 
Thus, for $n\ge 0$, we have 

\begin{Prop} \label{bern} 
\begin{equation}  
b^{(n)}(x)=(-1)^n\sum_{k=1}^{n+1}\frac{(k-1)!}{x^k}\left(\sts{n+1}{k}x-n\sts{n}{k}\right)b(x)^k\,.
\end{equation} 
\end{Prop}

\section{Zeta function}

Consider the zeta function defined by Mellin integral 
\begin{equation}  
Z_m(s, x)=\frac{1}{\Gamma(s)}\int_0^\infty t^{s-1}F^{(m)}(t)e^{-x t}dt\,,
\label{21}
\end{equation}  
where 
$$
F(t)=-f(-t)=\frac{1}{1-e^{-t}}\,, 
$$ 
which is defined for $\Re(s)>m+1$ and $\Re(x)>0$.  
It is not difficult, by use of the classical Riemann's methods, to prove its analytic continuation to whole complex plane $\C$ except simple poles at $s=1,2,\dots,m+1$.  We omit it.\\

On the other hand, we have the following results.  
\begin{theorem} 
Let $m$ be a nonnegative integer. Then we have 
$$
(-1)^m\sum_{k=1}^{m+1}(k-1)!\sts{m+1}{k}\zeta_k(s,x)=\sum_{k=0}^m(-1)^k\binom{m}{k}x^{m-k}\zeta(s-k,x)\,,
$$ 
where 
$$
\zeta(s,x)=\sum_{n=0}^\infty\frac{1}{(x+n)^s}
$$
denotes the Hurwitz zeta function and 
$$
\zeta_n(s,x)=\sum_{k_1,\dots,k_n=0}^\infty\frac{1}{(x+k_1+\cdots+k_n)^s}
$$
denotes the multiple Hurwitz zeta function of order $n$.  
\label{th1} 
\end{theorem}

Using the orthogonality  property (\ref{13}), from Theorem \ref{th1} we obtain 

\begin{Cor} 
For any positive integer $n$, we have 
\begin{align*} 
\zeta_n(s,x)&=\frac{(-1)^n}{(n-1)!}\sum_{k=0}^{n-1}(-1)^k\binom{n-1}{k}B_{n-k-1}^{(n)}(x)\zeta(s-k,x)\\
&=\frac{(-1)^n}{(n-1)!}\sum_{k=0}^{n-1}(-1)^k\sum_{m=0}^{n-k-1}\binom{m+k}{m}\stf{n}{m+k+1}x^m\zeta(s-k,x)\,,
\end{align*} 
where $B_j^{(n)}(x)$ is the $j$-th Bernoulli polynomial of order $n$.  
\label{cor1}
\end{Cor} 

\noindent 
{\it Proof of Theorem \ref{th1}.}  
By integration by parts, we have 
$$
Z_m(s,x)=\frac{1}{\Gamma(s)}\int_0^\infty(t^{s-1}e^{-x t})^{(m)}F(t)dt\,,
$$ 
where $F(t)=1/(1-e^{-t})$. Using the general Leibniz rule, we obtain 
\begin{equation} 
Z_m(s,x)=\sum_{k=0}^m\binom{m}{k}(-1)^{m-k}x^{m-k}\zeta(s-k,x)\,. 
\label{22}
\end{equation} 
On the other hand, from the relation (\ref{11}), we have 
\begin{align*} 
Z_m(s,x)&=\frac{1}{\Gamma(s)}\int_0^\infty t^{s-1}F^{(m)}(t)e^{-x t}dt\\
&=\sum_{k=1}^{m+1}(k-1)!\sts{m+1}{k}\frac{1}{\Gamma(s)}\int_0^\infty t^{s-1}\frac{e^{-x t}}{(1-e^{-t})^k}dt\,. 
\end{align*} 
Therefore, 
\begin{equation} 
Z_m(s,x)=\sum_{k=1}^{m+1}(k-1)!\sts{m+1}{k}\zeta_k(s,x)\,. 
\label{23}
\end{equation} 
The result follows from (\ref{22}) and (\ref{23}).  
\qed

\noindent 
{\it Remark.}  
From the analytic  continuations of $\zeta_n(s;x)$ and $\zeta(s;x)$  to the whole complex plane, the generalized $m$-th Bernoulli polynomial $B_m^{(n)}(x)$ of order $n$  and  the $m$-th Bernoulli polynomial $B_m(x)$  are related to $\zeta_n(-m;x)$ and $\zeta(-m;x)$ by the formulas 
\begin{eqnarray}
&&\zeta_n(-m;x)=(-1)^n\frac{m!}{(m+n)!}B_{m+n}^{(n)}(x),\label{interpolation1}\\
&&\zeta(-m;x)=-\frac{B_{m+1}(x)}{m+1},\label{interpolation2} 
\end{eqnarray}
for any non-negative integer $n$.  In special, if $s$ is a negative integer with $s=-m$ in Corollary \ref{cor1}, then we get the well-known formula \cite[ Corollary 1.8]{Bay-Be}: 
\begin{equation}  
B_{m+n}^{(n)}(x)=(m+n)\binom{m+n-1}{n-1}\sum_{k=0}^{n-1}(-1)^{k}\binom{m-1}{k}B_{n-k-1}^{(n)}(x)\frac{B_{m+k+1}(x)}{m+k+1}\,, 
\label{24}
\end{equation}   
which  recovers once more Euler's identity  and Dilcher's results in \cite{D}; it is also reminiscent of the convolution identities in \cite{pansun}. For more details, on various generalizations on Euler's identity for Bernoulli numbers, see \cite{ AD1,AD2,Chen,D,kamano,Min-Soo,Pet} and references therein.\\

Next, we study the zeta function associated with Mellin integral given by 
\begin{equation}  
\widehat Z_m(s,x)=\frac{1}{\Gamma(s)}\int_0^\infty t^{s-1}G^{(m)}(t)e^{-x t}dt\,,
\label{25}  
\end{equation}  
where 
$$
G(t)=b(-t)=\frac{t}{1-e^{-t}}\,. 
$$ 

We have the following.  

\begin{theorem}  
\begin{multline*} 
(-1)^m\sum_{k=1}^{m+1}(k-1)!\left(s\sts{m+1}{k}\zeta_k(s+1,x)-m\sts{m}{k}\zeta_k(s,x)\right)\\
=\sum_{k=0}^m\binom{m}{k}(-1)^k x^{m-k}(s-k)\zeta(s-k+1,m)\,.
\end{multline*} 
\label{th2}
\end{theorem} 

\begin{Cor}
If $s$ is a negative integer with $s=-n$, we obtain 
\begin{multline*} 
(-1)^m\sum_{k=0}^m\binom{m}{k}(-1)^k x^{m-k}B_{n+k}(x)\\
=\sum_{k=1}^{m+1}(-1)^{k-1}\frac{(k-1)!n!}{(n+k)!}\left(\sts{m+1}{k}(n+k)B_{n+k-1}^{(k)}(x)+m\sts{m}{k}B_{n+k}^{(k)}(x)\right)\,. 
\end{multline*} 
\label{cor2}
\end{Cor} 

In addition, if $x=0$, we obtain 
\begin{equation} 
B_{n+m}=\sum_{k=1}^{m+1}(-1)^{k-1}\frac{(k-1)!n!}{(n+k)!}\left(\sts{m+1}{k}(n+k)B_{n+k-1}^{(k)}+m\sts{m}{k}B_{n+k}^{(k)}\right)\,.
\label{27}
\end{equation}

\section{Zeta function of higher order}


Let $m_1,\dots,m_d$ be nonnegative integers. 
We study the zeta function $Z_{m_1,\dots,m_d}(s,x)$ defined by 
$$
Z_{m_1,\dots,m_d}(s,x)=\frac{1}{\Gamma(s)}\int_0^\infty t^{s-1}G^{(m_1)}(t)\cdots G^{(m_r)}(t)e^{-x t}dt
$$ 
($\Re(s)>m_1+\cdots+m_d+d$, $\Re(x)>0$). 

\begin{theorem}  
The function $Z_{m_1,\dots,m_d}(s,x)$ can be analytically continued to the whole $\mathbb C$ plane except simple poles at $s=1,2,\dots,m_1+\cdots+m_d+d$, and its values at non-positive integers are given as 
\begin{align} 
Z_{m_1,\dots,m_d}(-n,x)&=\bigl(B_{m_1}(x)+\cdots+B_{m_d}(x)\bigr)^n
\label{thd1}\\
&:=\sum_{k_1+\cdots+k_d=n\atop k_1,\dots,k_d\ge 0}\binom{n}{k_1,\dots,k_d}B_{m_1+k_1}(x)\cdots B_{m_d+k_d}(x)\,.\notag
\end{align} 
\label{th5}
\end{theorem}  
\begin{proof} 
Since 
\begin{align*} 
G^{(m)}(t)e^{-x t}&=(-1)^m b^{(m)}(-t)e^{-x t}\\
&=\sum_{n=0}^\infty\frac{(-1)^n B_{m+n}(x)}{n!}t^n\,,
\end{align*} 
we have 
\begin{align*} 
G^{(m_1)}(t)\cdots G^{(m_d)}(t)e^{-x t}&=\sum_{n=0}^\infty\left(\sum_{k_1+\cdots+k_d=n}(-1)^n\frac{B_{m_1+k_1}(x)\cdots B_{m_d+k_d}(x)}{k_1!\cdots k_d!}\right)t^n\\
&=\sum_{n=0}^\infty(-1)^n\left(\sum_{k_1+\cdots+k_d=n}\binom{n}{k_1,\dots,k_d}B_{m_1+k_1}(x)\cdots B_{m_d+k_d}(x)\right)\frac{t^n}{n!}\\
&=\sum_{n=0}^\infty(-1)^n\bigl(B_{m_1}(x)+\cdots+B_{m_d}(x)\bigr)^n\frac{t^n}{n!}\,.
\end{align*}  
Now, 
\begin{align*} 
&\int_0^\infty t^{s-1}G^{(m_1)}(t)\cdots G^{(m_d)}(t)e^{-x t}dt\\
&=\sum_{k=0}^\infty(-1)^k\frac{\bigl(B_{m_1}(x)+\cdots+B_{m_d}(x)\bigr)^k}{k!}\int_0^1 t^{s+k-1}e^{-xt}dt\\
&\quad +\sum_{k=0}^\infty(-1)^k\frac{\bigl(B_{m_1}(x)+\cdots+B_{m_d}(x)\bigr)^k}{k!}\int_1^\infty t^{s+k-1}e^{-xt}dt\,. 
\end{align*} 
The first integral converges if $\Re(s)>m_1+\cdots+m_d+d$ and 
the second integral converges absolutely for any $s\in\C$ and represents an
entire function of $s$. Hence, the right-hand side defines holomorphic
function for $\Re(s)>m_1+\cdots+m_d+d$. 
Since 
$$
\int_0^\infty t^{s+k-1}e^{-xt}dt=\frac{\Gamma(s+k)}{x^{s+k}}
$$ 
for $\Re(x)>0$, we obtain  
\begin{align*}  
Z_{m_1,\dots,m_d}(-n,x)&=\lim_{s\to -n}\left(\frac{1}{\Gamma(s)}\sum_{k=0}^\infty(-1)^k\frac{\bigl(B_{m_1}(x)+\cdots+B_{m_d}(x)\bigr)^k}{k!}\frac{\Gamma(s+k)}{x^{s+k}}\right)\\
&=\bigl(B_{m_1}(x)+\cdots+B_{m_d}(x)\bigr)^n\lim_{s\to -n}\frac{(-1)^n\Gamma(s+n)}{n!\Gamma(s)}\frac{1}{x^{s+n}}\\
&=\bigl(B_{m_1}(x)+\cdots+B_{m_d}(x)\bigr)^n\,. 
\end{align*} 
\qed\end{proof}

\begin{theorem}  
For nonnegative integers $k_1,\dots,k_d$ and $m_1,\dots,m_d$, we have 
\begin{multline} 
\sum_{m_1+\cdots+m_d=m}\binom{m}{m_1,\dots,m_d}Z_{m_1+k_1,\dots,m_d+k_d}(s,x)\\
=\sum_{k=0}^m\binom{m}{k}(-1)^k x^{m-k}Z_{k_1,\dots,k_d}(s-k,x)\,. 
\label{eq66} 
\end{multline}
\label{th6}
\end{theorem} 
\begin{proof}
By the general Leibniz rule, the left-hand side of (\ref{eq66}) is equal to 
\begin{align*}
&\frac{1}{\Gamma(s)}\int_0^\infty t^{s-1}\left(\sum_{m_1+\cdots+m_d=m}\binom{m}{m_1,\dots,m_d}G^{(m_1+k_1)}(t)\cdots G^{(m_d+k_d)}(t)\right)e^{-x t}dt\\
&=\frac{1}{\Gamma(s)}\int_0^\infty t^{s-1}\bigl(G^{(k_1)}(t)\cdots G^{(k_d)}(t)\bigr)^{(m)}e^{-x t}dt\,.
\end{align*} 
By integration by parts, 
$$
\int_0^\infty(t^{s-1}e^{-x t})^{(m)}G^{(k_1)}(t)\cdots G^{(k_d)}(t)dt
=(-1)^m\int_0^\infty t^{s-1}e^{-x t}\bigl(G^{(k_1)}(t)\cdots G^{(k_d)}(t)\bigr)^{(m)}dt\,. 
$$ 
On the other hand, we have 
\begin{align*}
&\int_0^\infty(t^{s-1}e^{-x t})^{(m)}G^{(k_1)}(t)\cdots G^{(k_d)}(t)dt\\
&=\sum_{k=0}^m\binom{m}{k}(-x)^{m-k}(s-1)_k\int_0^\infty t^{s-k-1}G^{(k_1)}(t)\cdots G^{(k_d)}(t)e^{-x t}dt\,. 
\end{align*} 
Hence, we get 
\begin{align*}
&\sum_{m_1+\cdots+m_d=m}\binom{m}{m_1,\dots,m_d}Z_{m_1+k_1,\dots,m_d+k_d}(s,x)\\
&=\sum_{k=0}^m\binom{m}{k}(-1)^k x^{m-k}\frac{(s-1)_k}{\Gamma(s)}\int_0^\infty t^{s-k-1}e^{-x t}dt\\
&=\sum_{k=0}^m\binom{m}{k}(-1)^k x^{m-k}Z_{k_1,\dots,k_d}(s-k,x)\,, 
\end{align*} 
which is (\ref{eq66}). 
\qed\end{proof}

\begin{Cor} 
We have 
\begin{equation} 
\sum_{m_1+\cdots+m_d=m}\binom{m}{m_1,\dots,m_d}Z_{m_1,\dots,m_d}(s,x)
=\sum_{k=0}^m\binom{m}{k}(-1)^k x^{m-k}\zeta_d(s-k,x)\,. 
\label{eq70} 
\end{equation}
\end{Cor}

Then by putting $s=-n$, we obtain the following.  

\begin{Cor}  
\begin{multline*} 
\sum_{m_1+\cdots+m_d=m}\binom{m}{m_1,\dots,m_d}\bigl(B_{m_1}(x)+\cdots+B_{m_d}(x)\bigr)^n\\
=\sum_{k=0}^m\binom{m}{k}(-1)^{k+d}x^{m-k}\frac{(n+k)!}{(n+k+d)!}B_{n+m+d}^{(d)}(x)\,. 
\end{multline*}  
In special, for $x=0$, we have 
$$
\sum_{m_1+\cdots+m_d=m}\binom{m}{m_1,\dots,m_d}\bigl(B_{m_1}+\cdots+B_{m_d}\bigr)^n
=\frac{(-1)^{m+d}(n+m)!B_{n+m+d}^{(d)}}{(n+m+d)!}\,. 
$$ 
\end{Cor}

\section{Another method and further formulas}
 
In this section, we study the zeta function 
$$
Z_{m_1,\dots,m_d}(s,x):=\frac{1}{\Gamma(s)}\int_0^\infty t^{s-1}G^{(m_1)}(t)\cdots G^{(m_d)}(t)e^{-x t}dt 
$$ 
by another method.    First, we consider the function $Z_{m_1,m_2}(s,x)$ when $d=2$.

\begin{Lem} 
\begin{equation}  
Z_{m_1,m_2}(s,x)=(-1)^{m_2}\sum_{0\le k_1\le k_2\le m_2}\binom{m_2}{m_2-k_2,k,k_2-k}(-x)^{k_2-k}Z_{m_1+m_2-k_2,0}(s-k,x)\,. 
\label{eq:502}
\end{equation} 
\label{lem502} 
\end{Lem} 
\begin{proof} 
By integration by parts and the general Leibniz rule, we have 
\begin{align*} 
Z_{m_1,m_2}(s,x)&=\frac{(-1)^m}{\Gamma(s)}\int_0^\infty\bigl(t^{s-1}e^{-x t}G^{(m_1)}\bigr)^{(m_2)}G(t)dt\\
&=(-1)^m\sum_{k_2=0}^{m_2}\binom{m_2}{k_2}\frac{1}{\Gamma(s)}\int_0^\infty(t^{s-1}e^{-x t})^{(k_2)}G^{(m_1+m_2-k_2)}(t)G(t)dt\,.  
\end{align*} 
Hence, 
\begin{align*} 
&Z_{m_1,m_2}(s,x)\\
&=(-1)^{m_2}\sum_{0\le k\le k_2\le m_2}\binom{m_2}{k_2}\binom{k_2}{k}(-x)^{k_2-k}\frac{1}{\Gamma(s-k)}\int_0^\infty t^{s-k-1}e^{-x t}G^{(m_1+m_2-k_2)}(t)G(t)dt\\
&=(-1)^{m_2}\sum_{0\le k_1\le k_2\le m_2}\binom{m_2}{m_2-k_2,k,k_2-k}(-x)^{k_2-k}Z_{m_1+m_2-k_2,0}(s-k,x)\,. 
\end{align*} 
\qed\end{proof}

By Proposition \ref{bern} and Lemma \ref{lem502}, we obtain the following theorem. 

\begin{theorem}  
We have 
\begin{align*}
Z_{m_1,m_2}(s,x)&=\sum_{0\le k\le k_2\le m_2\atop 0\le j\le m_1+m_2-k_2}(-1)^{j+m_2-1}(j-1)!\binom{m_2}{m_2-k_2,k_2-k,k}(-x)^{k_2-k}\\
&\quad\times\left(\sts{m_1+m_2-k_2+1,j}{j}(s-k)(s-k+1)\zeta_{j+1}(s-k+2,x)\right.\\
&\left.\qquad +(m_1+m_2-k_2)\sts{m_1+m_2-k_2}{j}(s-k)\zeta_{j+1}(s-k+1,x)\right)\,.
\end{align*}
\label{th510}
\end{theorem} 
\begin{proof} 
By Proposition \ref{bern}, we have 
\begin{align*} 
G^{(m_1+m_2-k_2)}(t)&=(-1)^{m_1+m_2-k_2}b^{(m_1+m_2-k_2)}(-t)\\
&=\sum_{j=1}^{m_1+m_2-k_2+1}\frac{(j-1)!(-1)^{j-1}}{t^j}\left(\sts{m_1+m_2-k_2+1}{j}t\right.\\
&\left.\quad+(m_1+m_2-k_2)\sts{m_1+m_2-k_2}{j}\right)B(-t)^j\\
&=\sum_{j=1}^{m_1+m_2-k_2+1}(j-1)!(-1)^{j-1}\left(\sts{m_1+m_2-k_2+1}{j}t\right.\\
&\left.\quad+(m_1+m_2-k_2)\sts{m_1+m_2-k_2}{j}\right)F(t)^j\,. 
\end{align*} 
Note that 
$$
G(t)=\frac{t}{1-e^{-t}}=t F(t)\,. 
$$ 
Therefore,  
\begin{align*} 
&Z_{m_1+m_2-k_2,0}(s-k,x)\\
&=\frac{1}{\Gamma(s-k)}\int_0^\infty t^{s-k-1}e^{-x t}G^{(m_1+m_2-k_2)}(t)G(t)\ dt\\
&=\sum_{j=1}^{m_1+m_2-k_2+1}(j-1)!(-1)^{j-1}\left(\sts{m_1+m_2-k_2+1}{j}\frac{1}{\Gamma(s-k)}\int_0^\infty t^{s-k+1}e^{-x t}F(t)^{j+1}\ dt\right.\\
&\quad\left.+(m_1+m_2-k_2)\sts{m_1+m_2-k_2}{j}\frac{1}{\Gamma(s-k)}\int_0^\infty t^{s-k}e^{-x t}F(t)^{j+1}\ dt\right)\\
&=\sum_{j=1}^{m_1+m_2-k_2+1}(j-1)!(-1)^{j-1}\left(\sts{m_1+m_2-k_2+1}{j}(s-k+1)(s-k)\zeta_{j+1}(s-k+2,x)\right.\\
&\quad\left.+(m_1+m_2-k_2)\sts{m_1+m_2-k_2}{j}(s-k)\zeta_{j+1}(s-k+1,x)\right)\,.
\end{align*} 
Together with Lemma \ref{lem502}, we get the desired result.  
\qed\end{proof}

Putting $s=-n$ in Theorem \ref{th510}, we obtain the following. 

\begin{Cor} 
We have 
\begin{align*}
&(B_{m_1}+B_{m_2})^n\\
&=\sum_{0\le k\le k_2\le m_2\atop 0\le j\le m_1+m_2-k_2}(-1)^{m_2}\binom{m_2}{k_2}(j-1)!\left(\sts{m_1+m_2-k_2+1,j}{j}\frac{(n+k_2)!B_{n+k_2+j-1}^{(j+1)}}{(n+k_2+j-1)!}\right.\\
&\left.\qquad -(m_1+m_2-k_2)\sts{m_1+m_2-k_2}{j}\frac{(n+k_2)!B_{n+k_2+j}^{(j+1)}}{(n+k_2+j)!}\right)\,.
\end{align*}
\end{Cor} 

\medskip

We find a very intriguing formula in the above corollary. Similarly, we can do the same for any $d\geq 3$ and obtain a generalization of the previous Theorem. However, the formulas obtained for $(B_{m_1}+\dots+B_{m_d})^n$ are not as pretty as in the case $d =2$. For this reason, we do not wish to state them here.


\end{document}